\documentclass[11pt]{amsart}

\usepackage{amsmath}
\usepackage{amssymb}
\usepackage{amsfonts}
\usepackage[mathscr]{eucal}
\usepackage{graphicx}
\usepackage{epstopdf}
\usepackage{verbatim}
\usepackage{amscd}
\usepackage{upgreek}
\usepackage{dsfont}
\usepackage{caption}
\usepackage{amsthm}
\usepackage{color}
\usepackage{hyperref}
\usepackage[all]{xy}
\usepackage{pdfsync}
\usepackage{array}

\usepackage{pgf,tikz}
\usepackage{mathrsfs}
\usetikzlibrary{arrows}

\usepackage{subcaption}

\usepackage{ccfonts,eulervm}
\usepackage[T1]{fontenc}

\makeindex
\numberwithin{equation}{section}

\newtheorem{Theorem}{Theorem}[section]

\newtheorem{Lemma}[Theorem]{Lemma}
\newtheorem*{CCL}{Cauchy's Combinatorial Lemma}
\newtheorem*{MainTheorem}{Main Theorem}

\theoremstyle{definition}

\newtheorem*{Remark}{Remark}

\usepackage{geometry}  % See geometry.pdf to learn the many layout options.
\usepackage[parfill]{parskip}  % Activate to begin paragraphs with empty 
\usepackage[percent]{overpic}   % to put overlay latex on images -- 2 stage 
\title{Almost all circle polyhedra are rigid}

\author{John C. Bowers}
\address{Department of Computer Science, James Madison University, 
Harrisonburg VA 22807}
\email{bowersjc@jmu.edu}

\author{Philip L. Bowers}
\address{Department of Mathematics, The Florida State University, 
Tallahassee FL 32306}
\email{bowers@math.fsu.edu}

\author{Kevin Pratt}
\address{Computer Science Department, Carnegie Mellon University, Pittsburgh PA 15213}
\email{kpratt@andrew.cmu.edu}

\date{\today} % Activate to display a given date or no date

\begin{document}

\begin{abstract}
We verify the infinitesimal inversive rigidity of almost all triangulated circle polyhedra in the Euclidean plane $\mathbb{E}^{2}$, as well as the infinitesimal inversive rigidity of tangency circle packings on the $2$-sphere $\mathbb{S}^{2}$. From this the rigidity of almost all triangulated circle polyhedra follows. The proof adapts Gluck's proof in~\cite{gluck75} of the rigidity of almost all Euclidean polyhedra to the setting of circle polyhedra, where inversive distances replace Euclidean distances and M\"obius transformations replace rigid Euclidean motions.
\end{abstract}

\maketitle

\section*{Introduction}
The infinitesimal rigidity theory of bar-and-joint frameworks traces its origins to the investigations of James Clerk Maxwell in 1864. Before Maxwell's investigations, Legendre and Cauchy investigated the global rigidity of polyhedra, culminating in the famous Cauchy Rigidity Theorem \cite{Cauchy1813} of 1813 that avers that convex, bounded polyhedra in Euclidean $3$-space $\mathbb{E}^{3}$ are globally rigid. In 1916 Max Dehn \cite{Dehn1916} generalized in one direction the Cauchy Theorem by proving the infinitesimal rigidity of strictly convex polyhedra in $\mathbb{E}^{3}$, and in the decade of the 1950's, A.D.~Alexandrov \cite{Alex2005} greatly extended Dehn's insights in his articulation of a rather broad theory of rigidity for $3$-dimensional polyhedra. Nonetheless, the question of whether or not all polyhedra, even non-convex ones, in $\mathbb{E}^{3}$ were rigid remained open. In 1975 Herman Gluck \cite{gluck75}, in the vein of Maxwell and Dehn, proved that almost all of them are infinitesimally rigid, and therefore rigid. At that time, many believed that all polyhedra in $\mathbb{E}^{3}$ were rigid, but in 1977 Robert Connelly \cite{Connelly1977} surprised the community by constructing a flexible polyhedron in $\mathbb{E}^{3}$, necessarily non-convex, and by Gluck's result, a rather rare example among polyhedra.

In \cite{BBP2018}, the authors of the present work began a study of the rigidity theory of circle frameworks, and in particular of circle polyhedra. There we showed that an analogue of the Cauchy Rigidity Theorem remains true when vertices in $\mathbb{E}^{3}$ are replaced by circles in the $2$-sphere $\mathbb{S}^{2}$ that are placed in the pattern of a Euclidean polyhedron. There global rigidity adheres when the circle polyhedron is convex and proper, and an example of Ma and Schlenker~\cite{Ma:2012hl} shows that this fails for non-convex ones, which mirrors the classical case for Euclidean polyhedra. Of course for circle polyhedra, Euclidean distance is replaced by inversive distance and rigidity is understood with respect to M\"obius transformations. In the present paper, we show that an analogue of Gluck's Theorem on the infinitesimal rigidity as well as the rigidity of almost all polyhedra holds in the setting of circle polyhedra in the plane $\mathbb{E}^{2}$. 

The proof, not surprisingly, mirrors the plan developed over the last couple of centuries by Cauchy \cite{Cauchy1813}, Dehn \cite{Dehn1916}, Alexandrov \cite{Alex2005}, and Gluck \cite{gluck75} in their investigations of rigid Euclidean polyhedra, with Gluck's paper in particular serving as a valuable guide for our development. The first task is identifying appropriate rigidity matrices and stress matrices in the setting of circle frameworks, where the trivial motions are by M\"obius transformations and inversive distance between circles is the desired preserved parameter. This is accomplished in Sections~\ref{Section:IR} and \ref{Section:IS}. With the right notions of these matrices in place, the argument proceeds as in Gluck~\cite{gluck75}. The infinitesimal rigidity of planar circle polyhedra is related to inversive stresses in Section~\ref{Section:RS}.  In the classical Euclidean development, Dehn's Theorem plays the role of showing that the collection of infinitesimally rigid polyhedra is non-empty, from which Gluck argues that the collection is dense in the space of all polyhedra. The role played by Dehn's Theorem in our development is of independent interest and appears in Section~\ref{sec:utcrigidity} where we show that univalent tangency circle packings of the $2$-sphere are infinitesimally rigid. The four preceding sections come together in the proof of the infinitesimal rigidity of almost all circle polyhedra in the final section, Section~\ref{Section:AAR}, as well as their rigidity. 

Because of the well-known, intimate connection between the hyperbolic geometry of the unit ball in $\mathbb{E}^{3}$ (the Beltrami-Klein model), the inversive geometry of its $2$-sphere ideal boundary $\mathbb{S}^{2}$, and the projective geometry of $\mathbb{RP}^{3}$ in which these models reside, (infinitesimal) rigidity results for circle polyhedra and frameworks have implications for the (infinitesimal) rigidity of hyperbolic polyhedra. These have been articulated in works of Thurston, Rivin, Hogdsen-Rivin, Bao-Bonahon, and our previous paper \cite{BBP2018}. We do not take the time to translate the results of this paper to the setting of, for example, strictly hyperideal hyperbolic polyhedra in $\mathbb{H}^{3}$ (see the final section of \cite{BBP2018}) as we are content with this mere mention of the connection.

\section{Infinitesimal Inversive Rigidity of Circle Frameworks}\label{Section:IR}

\subsection{Circle-frameworks and motions}
The general definition of a circle framework in the $2$-sphere allows for the case where the disks bounded by the circles of the framework cover the whole of the $2$-sphere.\footnote{For example, non-univalent tangency circle packings are circle frameworks, and their corresponding disks can cover the $2$-sphere multiple times.}  Our study will restrict this generality by considering only those circle frameworks that have at least one intersticial region, an open set not covered by the disks that the circles of the framework bound. This allows us to project stereographically to the plane and define the inversive rigidity matrix using the very simple Euclidean formul{\ae} for the inversive distances among the adjacent circles of the framework. Our study, then, is of circle frameworks in the plane. Nonetheless, our setting is general enough to cover all univalent circle frameworks in the $2$-sphere, the ones the interiors of whose companion disks are pairwise disjoint, as well as many with overlapping adjacent circles.

We parameterize the planar circles by their centers and radii and write $C = (x, y, r)$ when $C$ is the circle with center $(x, y)$ and radius $r$. A {\em motion} of $C$ is a continuously differentiable path $C(t) = (x(t), y(t), r(t))$ such that $C(0) = C$. Let $G$ be a connected graph on $n$ vertices labeled by the integers $1, \dots, n$ and with $m$ unoriented edges, each labeled by its pair of vertices. The set $\mathcal{C}=\{C_1, \cdots, C_n\}$ of circles in the plane indexed by the vertex set $V(G) = \{1, \dots, n\}$ is called a {\em circle framework with adjacency graph $G$}, or a \textit{\textit{c}-framework} for short, and is denoted as $G(\mathcal{C})$. Two circles $C_{i}$ and $C_{j}$ of the \textit{c}-framework $G(\mathcal{C})$ are \textit{adjacent} provided $ij$ is an edge of $G$. If all adjacent circles are tangent in $G(\mathcal{C})$, we call it a {\em tangency circle framework}, or a \textit{tc-framework}. If the closed disks in the collection $\mathcal{D} = \{ D_{1}, \dots , D_{n}\}$, where $\partial D_{i} = C_{i}$ for $i=1, \dots, n$, have pairwise disjoint interiors, then the circle framework $G(\mathcal{C})$ is said to be \textit{univalent}.

A {\em motion} of a framework $G(\mathcal{C})$ is given by a motion of each of the circles in $\mathcal{C}$ that at each time $t$ preserves inversive distances among adjacent circles:\:for all $t$ and each edge $ij\in E(G)$, the inversive distance between $C_i(t)$ and $C_j(t)$ remains constant. If there exists a smooth $1$-parameter family $M_t$ of M\"obius transformations such that $M_{0} = \text{id}_{\mathbb{R}^{2}}$ and $M_t(C_i) = C_i(t)$ for all $i\in V(G)$, then the motion is {\em trivial}; otherwise, the motion is a {\em flex} of $G(\mathcal{C})$. If there do not exist any flexes, then $G(\mathcal{C})$ is {\em rigid}, otherwise {\em flexible}. %An example of a flex is shown in fig.~\ref{fig:}. 

\subsection{Inversive rigidity matrix}
Let a motion of $G(\mathcal{C})$ be given as $$\mathcal{C}(t) = \{C_1(t) = (x_1(t), y_1(t), r_1(t)), \cdots, C_n(t) = (x_n(t), y_n(t), r_n(t))\}$$ and consider an edge $ij$. Differentiating the expression $$\text{Inv}(C_i(t), C_j(t)) = \frac{(x_{i}(t) - x_{j}(t))^2 + (y_{i}(t) - y_{j}(t))^2 - r_i(t)^2 - r_j(t)^2}{2 r_i(t) r_j(t)} = \text{constant},$$ multiplying by $-r_i^2 r_j^2$, letting $d_{ij}^2 = (x_{i} - x_{j})^2 + (y_{i} - y_{j})^2 = d_{ji}^{2}$, and evaluating at $t=0$ gives 

\begin{equation}\label{eq:infInvRigidity} 
\begin{split}
	& r_i r_j (x_{j} - x_{i}) x_i' + r_{j} r_{i} (x_{i} - x_{j}) x_j' + \\
	& r_i r_j (y_{j} - y_{i}) y_i' + r_{j} r_{i} (y_{i} - y_{j}) y_j' + \\
	& (1/2)r_j (r_i^2 + d_{ij}^2 - r_j^2) r_i' + \\
	& (1/2)r_i ( r_j^2 + d_{ji}^2 - r_i^2) r_j' = 0.
\end{split}
\end{equation}

Equation~\ref{eq:infInvRigidity} is linear in the derivatives $x_i'$, $y_i'$, $r_i'$, $x_j'$, $y_j'$, and $r_j'$. The system of linear equations over all edges $ij$ can be represented as a matrix equation $R c = \mathbf{0}$, where $R$ is an $m\times 3 n$ matrix corresponding to the coefficients of the linear system defined by Eq.~\ref{eq:infInvRigidity} and $c = (x_1', y_1', r_1', \cdots, x_n', y_n', r_n')^T$ is the column vector of derivatives. The rows are indexed by the $m$ edges and the columns by the $3n$ coordinates of the $n$ circles of the framework. When $ij$ is an edge of $G$, the row in $R$ corresponding to that edge has the entries 
\begin{equation}\label{EQ:rigiditymatrixentries}
r_i r_j (x_{j} - x_{i}), \quad r_i r_j (y_{j} - y_{i}), \quad (1/2)r_j (r_i^2 + d_{ij}^2 - r_j^2)
\end{equation}
in the three columns corresponding to the three coordinates of the circle $C_{i}$, and similarly for $C_{j}$ with $i$ and $j$ exchanged, and with all other entries of that row equal to zero. We call the matrix $R = R_{G(\mathcal{C})}$ the {\em inversive rigidity matrix} for $G(\mathcal{C})$. 

\subsection{Infinitesimal rigidity} The derivative of the motion $\mathcal{C}(t)$ at $t=0$ gives a solution vector $c$ to $R c= \mathbf{0}$. We call any vector $c$ satisfying $R c = \mathbf{0}$ an {\em infinitesimal  motion} of $G(\mathcal{C})$. Such an infinitesimal motion is {\em trivial} if it is the derivative of a trivial motion of $G(\mathcal{C})$, one by a smooth $1$-parameter family of M\"obius transformations. A non-trivial infinitesimal  motion is called an {\em infinitesimal  flex} of $G(\mathcal{C})$. If no infinitesimal  flex exists for $G(\mathcal{C})$, we say $G(\mathcal{C})$ is {\em infinitesimally  rigid}; otherwise, {\em infinitesimally  flexible}.

%\begin{Lemma}
%	Every infinitesimally  rigid \textit{c}-framework is  rigid. 
%\end{Lemma}
%\begin{proof} The proof proceeds as the proof of Theorem 4.1 of Gluck~\cite{}.	
%\end{proof}

The \textit{space of infinitesimal motions} of $G(\mathcal{C})$ is the kernel of $R$. Since trivial infinitesimal motions are derivatives of M\"obius transformations, a $6$-dimensional Lie group, describing them requires six parameters. Thus, the dimension of the kernel is at least six and equals six if and only if the framework is infinitesimally rigid. By the rank-nullity theorem, we have the following result.

\begin{Lemma}\label{lem:invrigidityrank}
$G(\mathcal{C})$ is infinitesimally rigid if and only if the rank of the inversive rigidity matrix $R$ is equal to $3n-6$. 
\end{Lemma}

\section{Inversive Stress of a \textit{c}-Framework}\label{Section:IS}

For the graph $G$ with vertex set $V(G) =\{ 1, \dots ,n \}$, the set $\mathfrak{p} = \{ p_{1}, \dots, p_{n}\}$ of points in $\mathbb{R}^{d}$ indexed by the vertex set $V(G) = \{1, \dots, n\}$ is called a \textit{Euclidean framework} and denoted as $G(\mathfrak{p})$.  A key concept that has found many uses in the rigidity theory of Euclidean frameworks is that of a stress on a framework. Informally, a stress is a real number $\omega_{ij}$ attached to each edge $ij$ of a framework and is used to compute force vectors $F_{ij} = - F_{ji}$ acting on $i$ and $j$ so that a positive stress attracts $i$ and $j$ towards each other, a negative stress repels, and a zero stress results in $F_{ij} = F_{ji} = 0$. For each vertex $i$ a force $F_i$ is computed by the sum $F_i = \sum_{j:ij\in E(G)} F_{ij}$. If for all $i$, $F_i = 0$, the stress is said to be an {\em equilibrium stress}. An equilibrium stress is {\em non-trivial} whenever there exists at least one edge $ij$ such that $F_{ij} \neq 0$. Alexandrov~\cite{Alex2005} proved that the only equilibrium stress on the edge-framework of a convex polyhedron in $\mathbb{R}^3$ is the trivial one if and only if the framework is infinitesimally rigid. Our aim in this section is to develop a similar notion of stress for \textit{c}-frameworks.  

\subsection{An inversive distance edge-vector}
We want to define stresses in such a way that an edge of a \textit{c}-framework with a non-zero stress generates an attractive or repellent force between its adjacent circles. The question is in what direction should the two circles move and how should their radii change? In the Euclidean case, a stressed edge attracts or repels its endpoints along the edge vector. One way to think of an edge vector in Euclidean space is as the direction in which the distance between two points increases or decreases most rapidly. The edge vector from $p_{i}$ to $p_{j}$ in $\mathbb{E}^{3}$ is $p_{j} - p_{i}$. Define the function $f_{i}$ as $f_{i}(p) = ||p - p_{i}||$, the distance function from all points $p\in \mathbb{E}^3$ to $p_{i}$. Then the gradient of $f_{i}$ at $p_{j}$ is equal to the normalized edge vector $\nabla f_{i}(p_{j}) = (p_{j} - p_{i}) / ||p_{j} - p_{i}||$. In physics, forces are usually the negative of the gradient of a potential. When thinking about the edge vector as representing a force, a more physical description might be that one should think of the point $p_{j}$ as exerting a force on the point $p_{i}$ determined by the potential function $f_{j}$, the distance from point $p_{j}$, so that the force is $-\nabla f_{j}(p_{i})$. Of course it is straightforward to see that $-\nabla f_{j}(p_{i}) = \nabla f_{i}(p_{j})$ so it doesn't matter which you use. It seems then that there are two natural candidates for the inversive distance edge-vector between two circles $C_i = (x_i, y_i, r_i)$ and $C_j = (x_j, y_j, r_j)$. Fixing the circle $C_{k}$ and denoting the inversive distance from $C_{k}$ to all other circles $C = (x, y, r)$ as
$$f_{k}(x,y,r) = \text{Inv}(C, C_{k}) = \frac{d_{k}^2 - r^2 - r_{k}^2}{2 r r_{k}},$$ 
where $d_{k}^2 = (x - x_{k})^2 + (y - y_{k})^2$, we might define the inversive distance edge-vector from vertex $i$ to vertex $j$ as either $\nabla f_{i}(x_{j}, y_{j}, r_{j})$ or as $-\nabla f_{j}(x_{i}, y_{i}, r_{i})$. These in general fail to be equal, so we have a choice to make. It turns out that using the more physically inspired expression works and we define the \textit{inversive distance edge-vector from circle $C_{i}$ to circle $C_{j}$} as $V_{ij} =-\nabla f_{j}(x_{i}, y_{i}, r_{i})$.

The partial derivatives of $f_{j}$ with respect to $x$, $y$, and $r$ are
\begin{equation}
%\begin{split}	
\frac{\partial f_{j}}{\partial x}  = \frac{x - x_{j}}{r r_{j}}, \quad
\frac{\partial f_{j}}{\partial y}  = \frac{y - y_{j}}{r r_{j}}, \quad
\frac{\partial f_{j}}{\partial r}  = \frac{r_{j}^{2}- d_{j}^2 - r^2}{2r^2 r_{j}}.
%\end{split}
\end{equation}
It follows that
$$V_{ij} = -\nabla f_{j}(x_{i}, y_{i}, r_{i}) = \left(\frac{x_{j} - x_{i}}{r_{i} r_{j}}, \frac{y_{j} - y_{i}}{r_{i} r_{j}},  \frac{r_{i}^{2} + d_{ij}^2 - r_{j}^2}{2r_{i}^{2} r_{j}}\right)^T.$$
Notice that unlike the standard case of Euclidean frameworks, the edge vectors are not negatives of one another since, in general, $V_{ij} \neq - V_{ji}$, this because of the asymmetry between $i$ and $j$ in the denominators of the third coordinates.

\subsection{Inversive stresses}

An {\em inversive stress} $\omega = \{\omega_{ij} : ij \in E(G)\}$ on a \textit{c}-framework $G(\mathcal{C})$ is an assignment of real numbers $\omega_{ij} = \omega_{ji}$ to each edge $ij \in E(G)$. For a fixed vertex $i$, each edge $ij$ adjacent to $i$ exerts a force $\omega_{ij} V_{ij}$ on vertex $i$. If the sum of these forces is 0, i.e., if
\begin{equation}\label{eq:inversiveequilibriumstress}
\sum_{j: ij \in E(G)} \omega_{ij} V_{ij} = 0,	
\end{equation}
 we say that the stress is {\em in equilibrium at vertex $i$}. If the stress is in equilibrium at all the vertices of a framework, then $\omega$ is an {\em equilibrium inversive stress} on $G(\mathcal{C})$. An inversive stress is \textit{non-trivial} if there exists at least one edge $ij$ for which $\omega_{ij} \neq 0$; otherwise it is {\em trivial}. 

The set of linear equations given by Eq.~\ref{eq:inversiveequilibriumstress} for $i = 1, \dots, n$ gives a linear system $V \omega = 0$. The matrix $V$ has dimension $3n \times m$  and its $ij$ column is given as follows:\;its $(3i-2, ij)$, $(3i-1, ij)$, and $(3i, ij)$ entries are the components of $V_{ij}$, its  $(3j-2, ij)$, $(3j-1, ij)$, and $(3j, ij)$ entries are the components of $V_{ji}$, and the remaining entries in the column are zero. Notice that there are exactly six possible non-zero entries in the $ij$th column, which are the components of the edge vectors $V_{ij}$ and $V_{ji}$ in the appropriate rows.

It follows that a non-trivial equilibrium inversive stress exists on $G(\mathcal{C})$ if and only if $\mathrm{dim}(\mathrm{ker} V) \neq 0$. We call $V$ the {\em inversive stress matrix} for $G(\mathcal{C})$. 

\begin{Remark} We may scale each vector $V_{ij}$ by an arbitrary non-zero real number $\lambda_{ij}=\lambda_{ji}$ without affecting the existence of a non-trivial equilibrium stress. To see this, suppose we have an inversive stress $\omega$ on $G(\mathcal{C})$. Choose any set of non-zero numbers $\Lambda = \{\lambda_{ij} : ij \in E(G)\}$, one for each edge $ij$. Now scale each vector $V_{ij}$ by its corresponding number $\lambda_{ij}$ to obtain the {\em scaled edge vector} $\widehat{V}_{ij} = \lambda_{ij} V_{ij}$. Let $\widehat{\omega}_{ij} = \omega_{ij} / \lambda_{ij}$. Then it follows immediately that 
\begin{equation}\label{eq:stressscaling}
\sum_{j: ij \in E(G)} \omega_{ij} V_{ij} = \sum_{j: ij \in E(G)} \widehat{\omega}_{ij} \widehat{V}_{ij},
\end{equation}
and $\omega_{ij} = 0$ if and only if $\widehat{\omega}_{ij} = 0$. In other words, the existence of a non-trivial equilibrium inversive stress on a framework $G(\mathcal{C})$ is independent of any choice of scale on the edge vectors $V_{ij}$. Eq.~\ref{eq:inversiveequilibriumstress} with the scaled edge vectors may be written as 
$\sum_{j: ij \in E(G)} \widehat{\omega}_{ij} \widehat{V}_{ij} = 0$, which gives the linear system $\widehat{V} \widehat{\omega} = 0$. We call the matrix $\widehat{V}$ a {\em scaled inversive stress matrix}.

We can state this rather nicely if we let $\Lambda = \text{diag} [ \omega_{ij} ]$ be the diagonal $m\times m$ matrix of rescale values rather than just the collection of rescale values. Then $\widehat{V} = V\Lambda$ and $\widehat{\omega} = \Lambda^{-1}\omega$, and Eq.~\ref{eq:stressscaling} becomes $V\omega = V\Lambda \Lambda^{-1} \omega = \widehat{V} \widehat{\omega}$. This implies in particular that the kernels $\ker V$ and $\ker\widehat{V}$ are isomorphic via the invertible matrix $\Lambda$. The next lemma summarizes this discussion.

\begin{Lemma}\label{lem:scaledstressequiv}
Let $V$ be the inversive stress matrix for $G(\mathcal{C})$ and let $\widehat{V} = V \Lambda$ be any scaled inversive stress matrix. Then $\dim (\ker V ) = \dim (\ker \widehat{V} )$. In particular, $\mathrm{dim}(\mathrm{ker}\, V) = 0$ if and only if $\mathrm{dim}(\mathrm{ker}\, \widehat{V}) = 0$ so that $G(\mathcal{C} )$ has a non-trivial inversive stress if and only if $\dim (\ker V) \neq 0$ if and only if $\dim(\ker \widehat{V} )\neq 0$.
\end{Lemma}

We use this fact with two different scalings in the next two sections. First, we choose one set of $\lambda_{ij}$-values to connect the concepts of inversive stresses and infinitesimal  rigidity, and then use a different set of $\lambda_{ij}$-values to prove that all \textit{tc}-frameworks are infinitesimally rigid. 
\end{Remark}

\section{Non-trivial Equilibrium Inversive Stresses and Infinitesimal Rigidity of Triangulated \textit{c}-Polyhedra}\label{Section:RS}

Before specializing to circle polyhedra, which are the \textit{c}-frameworks of interest in the remainder of the paper, we make some observations connecting the inversive rigidity matrix $R$ to the inversive stress matrix $V$. For each edge $ij$ of $G$, let $\lambda_{ij} = r_{i}^{2}r_{j}^{2}$ and let $\widehat{V}$ be the scaled inversive stress matrix scaled by the non-zero constants $\Lambda = \{ \lambda_{ij} : ij \in E(V)\}$. A moment's inspection reveals that 
\begin{equation}
\widehat{V}^{T} = R,	
\end{equation}
the transpose of the scaled inversive stress matrix is precisely the inversive rigidity matrix. In particular, the ranks of the matrices $\widehat{V}$ and $R$ agree, and by Lemma~\ref{lem:scaledstressequiv} so too does the rank of $V$. The next lemma now follows from Lemma~\ref{lem:invrigidityrank}.

\begin{Lemma}\label{lem:invrstressrank}
$G(\mathcal{C})$ is infinitesimally rigid if and only if the rank of the inversive stress matrix $V$ is equal to $3n-6$. 
\end{Lemma}

We now specialize to those \textit{c}-frameworks whose graphs are polyhedral graphs. A \textit{c}-framework $P(\mathcal{C})$ is a \textit{\textit{c}-polyhedron} if the graph $P$ is the 1-skeleton of an abstract triangulated polyhedron, equivalently, the $1$-skeleton of a simplicial triangulation of the $2$-sphere.\footnote{This is more general than the development of \textit{c}-polyhedra in~\cite{BBP2018} since we do not require the existence of ortho-circles for each face, as required there, and at the same time less general in that we do require the polyhedra to be triangulated, unlike there. The importance of the triangulation assumption is seen in the lemma following.} 

\begin{Lemma}\label{lem:equiv_inv_rigidity_stress}
A \textit{c}-polyhedron $P(\mathcal{C})$ in the plane is infinitesimally rigid if and only if it has no non-trivial equilibrium inversive stress. 
\end{Lemma}

\begin{proof}
Let $m$ be the number of edges and $n$ the number of vertices of $P$. Since $P$ is a triangulated polyhedron, the Euler characteristic gives $m = 3 n - 6$ so that the stress matrix $V$ for $P$ has dimension $3n \times (3 n - 6)$ and the rigidity matrix $R$ has dimension $(3 n - 6) \times 3n$. The result now is immediate. Indeed, $P(\mathcal{C})$ is infinitesimally rigid $\iff$ the rank of $R$ is $3n-6$ $\iff$ the rank of $R^{T} = \widehat{V}$ is $3n-6$ $\iff$ the $m = 3n-6$ columns of $\widehat{V}$ are linearly independent $\iff$ $\dim (\ker \widehat{V}) = 0$ $\iff$ $\dim (\ker {V}) = 0$ $\iff$ there is no non-trivial equilibrium inversive stress.
\end{proof}

Note that this proof works precisely because the number of edges $m$ is equal to $3n-6$ and the M\"obius group is a $6$-dimensional Lie group.

\section{Univalent Tangency packings of the Sphere are Infinitesimally Rigid}\label{sec:utcrigidity}

A \textit{tc-polyhedron} is a \textit{tc}-framework that is also a \textit{c}-polyhedron, so a configuration of circles in the pattern of a triangulated polyhedron whose adjacent circles are tangent. Recall that a \textit{c}-framework is univalent if the open disks bounded by the circles are pairwise disjoint.

\begin{Lemma}\label{lem:utcrigidity}
	Let $P(\mathcal{C})$ be a univalent \textit{tc}-polyhedron in the plane. Then $P(\mathcal{C})$ is infinitesimally rigid. 
\end{Lemma}
\begin{proof} In our argument we use Cauchy's Combinatorial Lemma, which Cauchy used in the proof if his celebrated rigidity theorem for convex Euclidean polyhedra. 
	
	\begin{CCL}
 Let $G$ be a graph that triangulates the $2$-sphere. Label some edges with $+$ sign, some with $-$ sign, and some not at all. Consider a vertex $v$ of the graph and a traversal of adjacent edges, say in counter-clockwise order. Ignoring the unlabeled edges, count the number of times a sign change occurs, from $-$ to $+$ or vice versa, in traversing in order the edges adjacent to $v$. It is easy to see that this count must be even. If there is a vertex of $G$ with at least two sign changes, then there must exist a vertex with exactly two sign changes. 		
	\end{CCL}
	
	To verify the lemma, it suffices by Lemma~\ref{lem:equiv_inv_rigidity_stress} to show that  $P$ has no non-trivial equilibrium inversive stresses. We argue by contradiction.
	
	Assume then that $P$ admits a non-trivial equilibrium stress $\omega$. Then $V \omega = 0$ and $\omega$ has a non-zero component for at least one edge. Let $\widehat{V}$ be the scaled inversive stress matrix obtained from $V$ by applying the scale $\lambda_{ij} = r_i r_j$. Then, 
	$$\widehat{V}_{ij} = \lambda_{ij} V_{ij} = \left(x_j - x_i\ , \ y_j - y_i\ ,\ \frac{-r_j^3 + d_{ij}^2 r_j + r_i^2 r_j}{2 r_i r_j}\right)^{T}.$$ 
	Since this is a tangency packing, $d_{ij}^2 = (r_i + r_j)^2$, and a quick calculation shows that the third component of $\widehat{V}$ simplifies to $r_{i} + r_{j}$ so that $\widehat{V}_{ij} = \left( x_j - x_i\, , y_j - y_i\, , r_i + r_j\right)^{T}$. Because $V \omega = 0$, $\widehat{V} \widehat{\omega} = 0$ by Lemma~\ref{lem:scaledstressequiv}, and because $\omega$ has at least one non-zero entry, so too does $\widehat{\omega}$. 
	
	 For any edge $ij \in E(V)$, label $ij$ with a $+$ sign if $\widehat{\omega}_{ij} > 0$ and a $-$ sign if $\widehat{\omega}_{ij} < 0$. We argue that there is no vertex $i$ that is incident to a vertex labeled with a $+$ or a $-$ sign such that all labeled edges incident to $i$ have the same label. To see this, treat $\widehat{V}$ as a vector in $\mathbb{R}^3$ and note that since all radii $r_i > 0$, the $z$-component of each vector $\widehat{V}_{ij}$ is positive. Suppose that for a vertex $i$ some of the edges $ij$ incident to $i$ have a $+$ label, but no incident edge is labeled with a $-$ sign. Then all of the $\widehat{\omega}_{ij}$ values are strictly positive. From this it follows that the $z$-component of the sum $\sum_{j\in Adj(i)} \widehat{\omega} \widehat{V}_{ij}$ is strictly positive, which contradicts that $\widehat{V}\widehat{\omega} = 0$. We conclude that any vertex $i$ that is adjacent to an edge labeled with a $+$ or a $-$ sign must give rise to at least two sign changes as one traverses in order the edges adjacent to $i$. Since $\widehat{\omega}$ has at least one non-zero entry, Cauchy's combinatorial lemma applies, and there exists a vertex $i$ with exactly two sign changes. 
	 
	 Let $v_1, \cdots, v_k$ denote the vertices adjacent to $i$ given in a counter-clockwise rotation. Without loss of generality assume that all the $+$ signs occur at the smaller indices, starting with $i v_1$. Let $i v_{s}$ be the last edge adjacent to $i$ labeled with a $+$ sign so that any of the remaining edges $i v_{s+1}, \dots, iv_{k}$, if labeled, are labeled with a $-$ sign. Consider the corresponding edge vectors $\widehat{V}_{i v_1}, \dots, \widehat{V}_{i v_k}$. Note first that each vector $\widehat{V}_{i v_{j}}$ lies on the $45^{\circ}$ cone $L :z^{2} = x^{2} + y^{2}$. This follows from the facts that circle $C_{i}$ is tangent to circle $C_{v_{j}}$ for $j= 1, \dots, k$, that $(x_{v_{j}} - x_{i})^{2} + (y_{v_{j}} - y_{i})^{2}$ is the squared-distance from the center of $C_{i}$ to that of $C_{v_{j}}$, and that $\widehat{V}_{iv_{j}} = \left( x_{v_{j}} - x_i\, , y_{v_{j}} - y_i\, , r_i + r_{v_{j}}\right)^{T}$.
	 
	 Assuming that circle $C_{i}$ is centered at the origin of the $xy$-plane, the orthogonal projection of the vector $\widehat{V}_{iv_{j}}$ is merely the position vector of the center of circle $C_{v_{j}}$. Since the \textit{tc}-polyhedron $P(\mathcal{C})$ is a univalent tangency circle packing, the centers of the circles $C_{v_{j}}$ appear in the order $1,\dots, k$ as one walks counterclockwise about the origin. It follows that there are two rays in the $xy$-plane, $R_{1}$ and $R_{2}$, such that one open sector of the plane bounded by the union $R_{1} \cup R_{2}$ contains the centers of $C_{v_{1}}, \dots, C_{v_{s}}$ and the complementary open sector contains the centers of the remaining circles $C_{v_{s+1}}, \dots C_{v_{k}}$. Let $\widehat{R}_{1}$ and $\widehat{R}_{2}$ be the orthogonal, vertical lifts of the respective rays $R_{1}$ and $R_{2}$ to the $45^{\circ}$ cone $L$, and let $\Pi$ be the plane in $\mathbb{E}^{3}$ containing these lifted rays. Then it is easy to see that the vectors $\widehat{V}_{i v_1}, \dots, \widehat{V}_{i v_{s}}$ lie in one open half-space bounded by $\Pi$ and the vectors $\widehat{V}_{i v_{s+1}}, \dots, \widehat{V}_{i v_{k}}$ lie in the complementary open half-space. Since $+$ signs may occur only at the edges $iv_{1}, \dots, iv_{s}$ and $-$ signs at edges $iv_{s+1}, \dots ,iv_{k}$ among the edges incident with vertex $i$, all the non-zero vectors in the list $\widehat{\omega}_{i v_{1}}\widehat{V}_{i v_1}, \dots, \widehat{\omega}_{i v_{k}}\widehat{V}_{i v_k}$ lie in the same open half-space bounded by $\Pi$. It follows that the sum $\sum_{j\in Adj(i)}\widehat{\omega}_{ij} \widehat{V}_{ij}$ is non-zero since there are exactly two sign changes about vertex $i$, which implies that not all the vectors $\widehat{\omega}_{i v_{1}}\widehat{V}_{i v_1}, \dots, \widehat{\omega}_{i v_{k}}\widehat{V}_{i v_k}$ are zero. This contradicts the fact that $\widehat{V} \widehat{\omega} =0$ and finishes the proof.
	 \end{proof}

\section{Almost All \textit{c}-Polyhedra are (Infinitesimally) Rigid}\label{Section:AAR}

The arguments of Gluck~\cite{gluck75}, adapted to the setting of \textit{c}-polyhedra, now apply to verify our main theorem. Here are the details.

Let $P$ be the $1$-skeleton of a simplicial triangulation of the $2$-sphere with vertices labeled as $1, \dots, n$. Identify the \textit{c}-polyhedron $P(\mathcal{C})$ where $\mathcal{C}$ is the circle collection $\{C_{1}, \dots, C_{n}\}$ with the point $p \in \mathbb{R}^{3n}$ whose coordinates are $p_{3i-2} = x_{i}$, $p_{3i-1} = y_{i}$, and $p_{3i} = r_{i}$ when $C_{i} = (x_{i}, y_{i}, r_{i})$ under our parameterization of circles. The collection of all \textit{c}-polyhedra in the pattern of $P$ is then parameterized by the points of the open subspace $\mathbb{O}$ of $\mathbb{R}^{3n}$ determined by the inequalities $r_{i} >0$ for $i= 1, \dots ,n$.

By Lemma~\ref{lem:invrigidityrank}, the \textit{c}-polyhedron $P(\mathcal{C})$ is infinitesimally flexible precisely when the rank of the inversive rigidity matrix $R = R_{P(\mathcal{C})}$ is less than $3n-6$. Since the rank of a matrix is the greatest integer $d$ such that some $d\times d$ sub-matrix has non-zero determinant, this occurs when every $(3n-6) \times (3n-6)$ sub-matrix of the rigidity matrix $R$ has zero determinant. By Equations~\ref{EQ:rigiditymatrixentries}, the coefficients of the rigidity matrix $R$ are polynomials in the coordinates of the parameter point $p\in  \mathbb{R}^{3n}$ that represents $P(\mathcal{C})$, and this implies that the determinant of any $(3n-6)\times (3n-6)$ sub-matrix of $R$ is a polynomial in the coordinates of $p$, i.e., in the variables $x_{i}, y_{i}, r_{i}$, $i=1, \dots, n$. It follows that the \textit{c}-polyhedron $P(\mathcal{C})$ is infinitesimally flexible if and only if the point $p$ representing $P(\mathcal{C})$ lies in the real algebraic variety $\mathbb{V}$ of $\mathbb{R}^{3n}$ determined by the polynomials $\det D = 0$, as $D$ ranges over the $(3n-6)\times (3n-6)$ sub-matrices of $R$. 

\begin{MainTheorem}
	The space $\mathbb{O} \setminus \mathbb{V}$ of parameter points corresponding to the infinitesimally rigid \textit{c}-polyhedra in the pattern of $P$ is open and dense in $\mathbb{O}$, and contains those parameter points corresponding to the rigid \textit{c}-polyhedra in the pattern of $P$. 
\end{MainTheorem}

\begin{proof}
	To prove the first statement, that the space $\mathbb{O}\setminus \mathbb{V}$ is open and dense in $\mathbb{O}$, it suffices to show that $\mathbb{V}$ is a proper subvariety of $\mathbb{R}^{3n}$. For this we need but demonstrate the existence of a single \textit{c}-polyhedron $P(\mathcal{C})$ that is infinitesimally rigid. The Koebe Circle Packing Theorem\footnote{This was proved first in Koebe~\cite{pk36}, rediscovered by Thurston~\cite{Thurston:1980}, and now is a part of what is known as the Koebe-Andre'ev-Thurston  Theorem. There are many proofs in the literature. See Bowers~\cite{Bowers2018} for a fairly quick proof and a survey of results relating to the theorem, as well as a bibliography of literature surrounding the theorem.} implies the existence of a univalent, tangency circle packing of the $2$-sphere in the pattern of $P$, which then stereographically projects to give a univalent \textit{tc}-polyhedron $P(\mathcal{C})$ in the plane. Lemma~\ref{lem:utcrigidity} implies that $P(\mathcal{C})$ is infinitesimally rigid. 
	
	For the second part, that the infinitesimal rigidity of a \textit{c}-polyhedron implies its rigidity, we follow almost exactly the proof of Theorem 4.1 of Gluck~\cite{gluck75}. It deserves to be separated out as its own theorem, whose proof will finish off the proof of this Main Theorem.
\end{proof}

\begin{Theorem}
If the \textit{c}-polyhedron $P(\mathcal{C})$ is infinitesimally rigid, it is rigid.	
\end{Theorem}
\begin{proof}
As in Gluck~\cite{gluck75}, the proof is an application of the implicit function theorem. It works in the present setting because the number of edges $m$ of $P$ is $3n-6$ by an Euler characteristic argument, since $P$ is the $1$-skeleton of a triangulation of $\mathbb{S}^{2}$, and the M\"obius group is a $6$-dimensional Lie group. Here are the details.

The map $f:\mathbb{O} \to \mathbb{R}^{m = 3n-6}$ that assigns the inversive distances between adjacent circles in the \textit{c}-polyhedron $P(\mathcal{C})$ corresponding to $p \in \mathbb{O}$ via
\begin{equation*}
f(p)_{ij} = \text{Inv} ((C_{i},C_{j})) \quad\text{when}\quad ij \in E(P)
\end{equation*}
has derivative
\begin{equation*}
df_{p}	= A_{p} R_{P(\mathcal{C})},
\end{equation*}
where $p$ corresponds to $P(\mathcal{C})$ and $A_{p}$ is the diagonal matrix whose $ij$ diagonal entry is $-1/r_{i}^{2}r_{j}^{2}$ when $ij \in E(P)$. As $A_{p}$ is invertible, we have $P(\mathcal{C})$ is infinitesimally rigid $\iff$ $R=R_{P(\mathcal{C})}$ has rank $3n-6$ $\iff$ $R$ is surjective $\iff$ $df_{p}$ is surjective $\iff$ $p$ is a regular value of $f$. By the implicit function theorem, $f^{-1}f(p)$ is a $6$-dimensional manifold near $p$. A moment's consideration shows that $f^{-1}f(p)$ parameterizes the set of \textit{c}-polyhedra in the pattern of $P$ with the same inversive distances between adjacent circles as those of $P(\mathcal{C})$. Let $\text{M\"ob}(p)$ denote the set of parameter points $q$ that correspond to the \textit{c}-polyhedra in the orbit of $P(\mathcal{C})$ under the action of the M\"obius group on the extended plane. Since the Lie group of M\"obius transformations is $6$-dimensional, $\text{M\"ob}(p)$ is a $6$-dimensional manifold near $p$. Since $\text{M\"ob}(p)\subset f^{-1}f(p)$ and both are $6$-dimensional manifolds near $p$, they coincide in a neighborhood of $p$. But this says precisely that any motion of the \textit{c}-polyhedron $P(\mathcal{C})$ is in fact trivial. Hence $P(\mathcal{C})$ is rigid.
\end{proof}

\bibliographystyle{plain}

\end{document}